\documentclass[reqno,twoside,11pt]{amsart}

\usepackage{amsmath, amsfonts, amssymb, esint, amsthm, nicefrac, bbm, dsfont, enumitem,}
\usepackage{epsfig, graphicx}
\usepackage{hyperref, cite}
\usepackage{xcolor} 

\def\sideremark#1{\ifvmode\leavevmode\fi\vadjust{\vbox to0pt{\vss
 \hbox to 0pt{\hskip\hsize\hskip1em
\vbox{\hsize2cm\small\raggedright\pretolerance10000
 \noindent #1\hfill}\hss}\vbox to8pt{\vfil}\vss}}}

\setlist[itemize]{leftmargin=*}
\DeclareMathAlphabet{\mathpzc}{OT1}{pzc}{m}{it}
\newcommand{\Lp}{L_{\phi}}
\newcommand{\Lse}{\mathcal{L}_s^\epsilon[\phi]}
\newcommand{\Ls}{\mathcal{L}_s[\phi]}
\newcommand{\dm}{\;\mathrm{d}\mu_s(t)}
\newcommand{\dmN}{\;\mathrm{d}\mu_s^N(z)}
\newcommand{\R}{\mathbb{R}}

\newcommand{\A}{\mathcal{A}}

\renewcommand{\epsilon}{\varepsilon}

\numberwithin{equation}{section}

\theoremstyle{plain}
\newtheorem{theorem}{Theorem}[section]
\newtheorem{lemma}[theorem]{Lemma}
\newtheorem{corollary}[theorem]{Corollary}
\newtheorem{proposition}[theorem]{Proposition}

\theoremstyle{definition}
\newtheorem{remark}[theorem]{Remark}
\newtheorem{definition}[theorem]{Definition}

\begin{document}
\title[Asymptotic expansions for the fractional infinity Laplacian]
{On asymptotic expansions for the fractional infinity Laplacian} 
\author{F\'elix del Teso, J\o rgen Endal, and Marta Lewicka}
\address{F. del Teso: Universidad Complutense de Madrid, Departamento de An\'alisis Matem\'atico y Matem\'atica
Aplicada, 28040 Madrid, Spain}  
\address{J. Endal: Universidad Aut\'onoma de Madrid, Departamento de Matem\'aticas, 28049 Madrid, Spain} 
\address{M. Lewicka: University of Pittsburgh, Department of
  Mathematics, 139 University Place, Pittsburgh, PA 15260, USA} 
\email{fdelteso@ucm.es, jorgen.endal@uam.es, lewicka@pitt.edu} 
	
\thanks{F. del Teso was partially supported by PGC2018-094522-B-I00 from the MICINN of the Spanish Government.\\ 
J. Endal received funding from the Research Council of Norway under the Toppforsk (research excellence) grant agreement no. 250070 ``Waves and Nonlinear Phenomena (WaNP)'', from the European Union’s Horizon 2020 research and innovation programme under the Marie Sk{\l}odowska-Curie grant agreement no. 839749 ``Novel techniques for quantitative behavior of convection-diffusion equations (techFRONT)'', and from the Research Council of Norway under the MSCA-TOPP-UT grant agreement no. 312021.\\
M. Lewicka was supported by NSF grants DMS-1613153 and DMS-2006439.}



\begin{abstract}
We propose two asymptotic expansions of two interrelated
integral-type averages, in the context of the fractional $\infty$-Laplacian
$\Delta_\infty^s$ for $s\in (\frac{1}{2},1)$. This operator has been introduced and first studied
in [Bjorland, C., Caffarelli, L. and Figalli, A., \emph{Nonlocal Tug-of-War and the inifnity fractional Laplacian}, Comm. Pure Appl. Math., {\bf 65}, pp. 337--380,  (2012)]. Our expansions are parametrised by the radius of the
removed singularity $\epsilon$, and allow for the identification of $\Delta_\infty^s\phi(x)$ as the
$\epsilon^{2s}$-order coefficient of the deviation of the $\epsilon$-average from
the value $\phi(x)$, in the limit $\epsilon\to
0+$. The averages are well posed for functions $\phi$ that are only Borel regular and bounded.
\end{abstract}

\maketitle

\section{Introduction}

This paper concerns the fractional $\infty$-Laplace
operator $\Delta_\infty^s$, as introduced in \cite{BCF} and represented by (\ref{inno}) below. Given a
function $\phi:\R^N\to\R$, our main result is the identification of $\Delta_\infty^s \phi(x)$
as the $\epsilon^{2s}$-order coefficient in the asymptotic expansion of the deviation
of an appropriate $\epsilon$-average $\mathcal{A}_\epsilon$ applied on $\phi$, from the value $\phi(x)$.
Such identification is of general interest in the analysis of partial differential
operators, their related probabilistic interpretation via Tug-of-War
games, a study of viscosity solutions and of numerical approximating schemes. The chief
example of the said asymptotic expansions is given by the well known (local and linear)
formula for the Laplace operator, where $\Delta \phi(x)$ emerges as the
$\epsilon^2$-order coefficient from the integral average $ \fint_{B_\epsilon}$:
\begin{equation}\label{aed2}
\fint_{B_\epsilon(x)}\phi(y)\;\mbox{d}y = \phi(x) +
\frac{\epsilon^2}{2(N+2)}\Delta \phi(x) + o(\epsilon^2) \qquad \mbox{ as }\; \epsilon\to 0+. 
\end{equation}
The parallel expansion of the $\infty$-Laplacian: $\Delta_\infty\phi(x) =
\big \langle \nabla^2 \phi(x) : \frac{\nabla\phi(x)}{|\nabla\phi(x)|}\otimes 
\frac{\nabla\phi(x)}{|\nabla\phi(x)|}\big\rangle$ utilizes the
midpoint (local and nonlinear) average $\frac{1}{2}(\sup_{B_\epsilon} + \inf_{B_\epsilon})$ in:
\begin{equation}\label{aed3}
\frac{1}{2} \big(\sup_{B_\epsilon(x)\normalcolor}\phi + \inf_{B_\epsilon(x)}\phi \big) = \phi(x) +
\frac{\epsilon^2}{2}\Delta_\infty \phi(x) + o(\epsilon^2) \qquad \mbox{ as }\; \epsilon\to 0+. 
\end{equation}

\subsection{The asymptotic expansions and averaging operators in this paper}
In what follows,  will prove that for every $s\in (\frac{1}{2},1)$ one
counterpart formula of (\ref{aed2}) for $\Delta_\infty^s$ is: 
\begin{equation}\label{jeden1}
\begin{split}
& \A_\epsilon^o \phi(x) = \phi(x) +
{s}\epsilon^{2s}\Delta_\infty^s\phi (x) + o(\epsilon^{2s}) \qquad \mbox{ as }\; \epsilon\to 0+,
\end{split}
\end{equation}
based on the following (nonlocal and nonlinear) average:
\begin{equation}\label{MVP1}
\A_\epsilon^o \phi(x) = \frac{1}{2}\Big(\sup_{|y|=1}\fint_\epsilon^\infty
\phi(x+ty)\dm + \inf_{|y|=1}\fint_\epsilon^\infty \phi(x+ty)\dm \Big).
\end{equation}
The one-dimensional fractional measure $\mu_s$ and the structure of
the error term $o(\epsilon^{2s})$ will be explained below.
When $\nabla\phi(x)\neq 0$, we also derive another identification through
a local-nonlocal average, which is a convex combination of
the averages used in (\ref{jeden1}) and (\ref{aed3}):
\begin{equation}\label{MVP2}
\A_\epsilon \phi(x) = (1-s) \cdot \A_\epsilon^o \phi(x) + s\cdot \frac{1}{2}
\big(\sup_{B_\epsilon(x)\normalcolor}\phi + \inf_{B_\epsilon(x)}\phi \big).
\end{equation}
We anticipate that the error quantity $o(\epsilon^{2s})$ below is
uniform in the whole considered range $s\in
(\frac{1}{2},1)$, whereas the corresponding error in (\ref{jeden1})
blows up to $\infty$ as $s\to 1-$.
Thus, the following asymptotic expansion can be seen as an improvement of (\ref{jeden1}):
\begin{equation}\label{dwa1}
\A_\epsilon\phi(x) = \phi(x) + (1-s)s\epsilon^{2s}
\Delta_\infty^s \phi(x) + o(\epsilon^{2s})\qquad \mbox{ as }\; \epsilon\to 0+.
\end{equation}
Precise statements of (\ref{jeden1}) and (\ref{dwa1})
will be given in Theorems \ref{jaszczu1}, \ref{jaszczu2} and Remarks
\ref{kura}, \ref{piesio}.

\subsection{The fractional $\infty$-Laplacian}
Let $\phi:\R^N\to\R$ be a bounded Borel function. We recall that
$\phi\in C^{1,1}(x)$ at $x\in\R^N$, provided that there exists $p_x\in \R^N$ and $C_x,\eta_x>0$ such that:
\begin{equation}\label{C11}
\big|\phi(x+y) - \phi(x) - \langle p_x, y\rangle\big|\leq C_x |y|^2
\qquad\mbox{for all }\; |y|<\eta_x.
\end{equation}
In \cite[Definition 1.1]{BCF},  the (normalized) fractional
$\infty$-Laplacian $\Delta_\infty^s\phi(x)$, for $s\in (\frac{1}{2}, 1)$, has been introduced by means of
two distinct formulas, distinguishing between cases $p_x=0$ and
$p_x\neq 0$.  In section \ref{sec2}, we provide a rigorous proof of the following alternative definition stated in \cite{BCF}: 
\begin{equation}\label{inno}
\Delta_\infty^s \phi(x)= \frac{1}{\alpha_s}\cdot\sup_{|y|=1}\inf_{|\tilde y|=1} \int_0^\infty \Lp(x,ty,t\tilde y)\dm.
\end{equation}
To explain the notation in the right hand side above, for each $x, y, \tilde y\in\R^N$ we define: 
$$\Lp(x, y, \tilde y) \doteq \phi(x+y) + \phi(x-\tilde y) - 2\phi(x).$$ 
Further, $\mu_s$ is the measure\footnote{The role of the normalizing constant $\alpha_s$ is to ensure that the 
operator $-(-\Delta)^su(x) \doteq \int_0^\infty L_u(x,t,t)\dm$ defined for 
$u:\R\to\R$, is a pseudo-differential operator with symbol
$|\xi|^{2s}$.} on the Borel subsets of $(0,\infty)$, given by:
$$\mathrm{d}\mu_s(t) \doteq \frac{\alpha_s}{t^{1+2s}}\;\mbox{d}t \quad \mbox{ where }\;
\alpha_s=\frac{4^s s\Gamma\big(\frac{1}{2}+s\big)}{\pi^{1/2}\Gamma\big(1-s\big)} =
\Big(2\int_0^\infty\frac{1-\cos t}{t^{1+2s}}\;\mbox{d}t\Big)^{-1} .$$
It is important to note \cite{hitch}  that one can express $\alpha_s$ by means of another constant\footnote{A direct
  calculation shows that, for example, $c_s\in ((\frac{12}{13})^2,
  (\frac{12}{5})^2)$ in the range $s\in (\frac{1}{2}, 1)$.} $c_s$, that
is bounded and positive, uniformly in $s$. Namely, there holds:
$$\alpha_s = s(1-s) c_s.$$

We also point out that the operator $\Delta_\infty^s$ treated in this
paper, is not the only nonlocal counterpart of $\Delta_\infty$. A
variational (i.e. energy based) fractional infinity Laplacian was
studied in \cite{CLM12} and a non normalized one in
\cite{CJ17}. These operators are not suited for a game theoretical
approach, which was the main motivation  in\cite{BCF}.

\subsection{Statements and discussion of main results}
We consider the operator in the right hand side of (\ref{inno}):
\begin{equation}\label{Ldefi1}
\Ls (x) \doteq\sup_{|y|=1}\inf_{|\tilde y|=1} \int_0^\infty \Lp(x,ty,t\tilde y)\dm.
\end{equation}
Given $\eta_x>0$, we work with the following hypotheses on
$\phi$, relative to the ball $B_{\eta_x}\doteq B_{\eta_x}(x)$: 
\begin{equation*}\label{H1}\tag{{\bf H}}
\left[\quad \;\; \mbox{\begin{minipage}{13.1cm}\vspace{1mm}
\begin{itemize}[leftmargin=*]
\item[(i)] $\phi\in C^2(\bar B_{\eta_x})$ where $p_x\doteq \nabla
\phi(x)$ and $C_x\doteq \frac{1}{2}\|\nabla^2\phi\|_{L^\infty(B_{\eta_x})}$.\vspace{1mm}
\item[(ii)] $\phi$ is bounded and uniformly continuous on
$\R^N\setminus \bar B_{\eta_x}$ with modulus of continuity
$\omega_\phi (a)\doteq\sup\big\{|\phi(y_1)-\phi(y_2)|;~ y_1,y_2\in \R^N\setminus
\bar B_{\eta_x},~ |y_1-y_2|\leq a\big\}.$\vspace{1mm}
\end{itemize}
\end{minipage}}\;\;\, \right]
\end{equation*}
Regularity required in (\ref{H1}) is  satisfied by the test functions in the viscosity solution setting (see section
\ref{sec51}). We further denote: 
\begin{equation*}
\begin{split}
& A_\epsilon = \max\Big\{\frac{16 C_x}{|p_x|}\cdot\frac{2s-1}{1-s}\cdot 
\frac{\eta_x^{2-2s}-\epsilon^{2-2s}}{\epsilon^{1-2s}
  -{\eta_x}^{1-2s}}, ~\kappa_\epsilon\Big\}, \\ &
\kappa_\epsilon = \sup\Big\{ a;~ a\in [0,2] \; \mbox{ and } \; a^2\leq
 \frac{8\omega_\phi(a)}{|p_x|}\cdot \frac{\frac{2s-1}{2s}
  \eta_x^{-2s} +\eta_x^{1-2s}}{\epsilon^{1-2s} - \eta_x^{1-2s}}\Big\}.
\end{split}
\end{equation*}

Our first main result regards the expansion (\ref{jeden1}):
\begin{theorem}\label{jaszczu1}
Let $\phi:\R^N\to\R$ satisfy (\ref{H1}).  Then there holds, for every $\epsilon< \eta_x$:
\begin{equation*}
\begin{split}
\Big|\A_\epsilon^o&\phi(x) - \phi(x) -
\frac{1}{c_s(1-s)}\epsilon^{2s}\Ls (x) \Big|\leq  \frac{s}{1-s}\cdot
C_x \epsilon^2 + \\ & + \left\{\begin{array}{ll} \epsilon^{2s}\Big(4
sC_x\frac{\eta_x^{2-2s}-\epsilon^{2-2s}}{1-s}\cdot A_\epsilon
+ \big(\eta_x^{-2s} + \frac{2s}{2s-1}\eta_x^{1-2s}\big)\cdot \omega_\phi(A_\epsilon)\Big)
& \mbox{ when }\; p_x\neq 0\\
0 & \mbox{ when }\; p_x= 0. \end{array}\right.
\end{split}
\end{equation*}
\end{theorem}

Our second main result regards the expansion (\ref{dwa1}):
\begin{theorem}\label{jaszczu2}
Assume (\ref{H1}) and $p_x\neq 0$.
Then, for all $\epsilon< \eta_x$ with $\epsilon |\nabla^2\phi(x)|\leq |p_x|$, we have:
\begin{equation*}
\begin{split}
\Big|\A_\epsilon\phi(x) &- \phi(x) - \frac{1}{c_s}\epsilon^{2s}\Ls (x)
\Big|\\ & \leq 
2\epsilon^{2s}\Big(2sC_x\big(\eta_x^{2-2s}-\epsilon^{2-2s}\big)A_\epsilon
+ \Big(\frac{\eta_x^{-2s}}{2}+\frac{s\eta_x^{1-2s}}{2s-1}\Big)\cdot (1-s) \omega_\phi(A_\epsilon)\Big)
\\ & \quad + 2s\epsilon^3\frac{|\nabla^2\phi(x)|^2}{|p_x|}
+ s\epsilon^2\sup_{y\in B_\epsilon(x)}\big|\nabla^2\phi(y)-\nabla^2\phi(x)\big|.
\end{split}
\end{equation*}
\end{theorem}
For a discussion of the error terms in the above results, we refer to Remarks \ref{kura} and
\ref{piesio}. In particular, for $\phi$ Lipschitz, the
 bound in Theorem \ref{jaszczu1} becomes: $O(\epsilon^{4s-1}+\epsilon^{2})$ with constants that blow up as
$s\to1-$. On the other hand, the bound in Theorem \ref{jaszczu2} has
the form: $O(\epsilon^{4s-1}) +  o(\epsilon^2)$ however the related
constants are uniform in $s\in(\frac{1}{2},1)$, and it also is
compatible with the expected error bound for the (local) $\infty$-Laplacian. This improvement is obtained
by correcting the singular part of $\Delta_\infty^s$ by the
corresponding asymptotic expansion of its local counterpart. Such
idea was already present in the numerical analysis literature, where
it was used to obtain higher order monotone numerical schemes for the
fractional Laplacian and other linear nonlocal operators, see
e.g.\cite{DEJ18, DEJ19}. The fact that the singular part of a
nonlinear and nonlocal operator encodes a local counterpart 
is an idea also present in \cite{BS2,DL20}. 

Other asymptotic expansions for nonlocal operators such as
$\Delta_\infty^s$ have been recently introduced  in
\cite{BS1, BS2}. The related average in \cite[section 3.2]{BS2} distinguishes
between the cases $p_x\neq 0$ and $p_x=0$. In comparison,
$\mathcal{A}_\epsilon$ in the present paper neither
relies on this distinction nor even necessitates the notion of the
gradient being well posed. Thus, they can be applied on a larger class of functions $\phi$ that are
only bounded Borel. 

In section \ref{sec52}, we further propose a version $\bar{\mathcal{A}}_\epsilon^o$ of the average
$\mathcal{A}_\epsilon^o$ and its corresponding expansion, in which integration takes place on an open, bounded
domain in $\R^N$, rather than an infinite line. We believe that this
correction will be of importance in the implementation of
numerical schemes. We also conjecture that the expected values of the
stochastic process whose dynamic
programming principle is modeled on $\bar{\mathcal{A}}^o_\epsilon$ converge
to these solutions in the limit $\epsilon\to 0+$, as in the
pivotal study \cite{PSSW} of  the classical operator $\Delta_\infty$.

\subsection*{Outline of the paper}
We prove (\ref{inno}) in section
\ref{sec2}, Theorem \ref{jaszczu1} in section \ref{sec3}, and Theorem
\ref{jaszczu2} in section \ref{sec4}. In section \ref{sec5}, we
discuss $\bar{\mathcal{A}}_\epsilon^o$ and  put Theorems
\ref{jaszczu1}, \ref{jaszczu2} in a viscosity solution framework.

\section{The fractional $\infty$-Laplacian and a proof of (\ref{inno})}\label{sec2}
 
Given a bounded Borel function $\phi:\R^N\to\R$, and two parameters
$\epsilon>0$ and $s\in (\frac{1}{2},1)$, we will be concerned
with values of the integral operators $\mathcal{L}_s^\epsilon[\phi]:\R^N\to\R$, given in:
\begin{equation}\label{Ldefi}
\begin{split}
\Lse (x) & \doteq\sup_{|y|=1}\inf_{|\tilde y|=1} \int_\epsilon^\infty \Lp(x,ty,t\tilde y)\dm\\
& = \sup_{|y|=1}\int_\epsilon^\infty\phi(x+ty)\dm + 
\inf_{| y|=1} \int_\epsilon^\infty\phi(x+ty)\dm - \frac{(1-s)c_s}{\epsilon^{2s}}\phi(x),\\
\end{split}
\end{equation}
Note that, since the restriction of $\phi$ to any
one-dimensional line is also Borel, the function $(0,\infty)\ni
t\mapsto \Lp(x,ty,t\tilde y)$ is bounded and Borel for any $x,y,\tilde y$. 
Further, since $\mu_s(\epsilon, \infty) = \frac{\alpha_s}{2s\epsilon^{2s}}<\infty$, each integral
$\int_\epsilon^\infty \Lp(x, ty, t\tilde y)\dm$ and consequently also the quantities
$\Lse (x)$, are all well defined and finite. On the other hand,
$\mu_s(0,\infty)=\infty$, so neither the definition of $\Ls$ 
in (\ref{Ldefi1}) nor a version of its equivalent formulation as in $\Lse$ are 
necessarily valid, when $\phi$ is only bounded and Borel. 
However, one immediate consequence of (\ref{C11}) is that:
\begin{equation}\label{C2}
\big|\Lp (x,ty,t\tilde y) - t\langle p_x, y-\tilde y\rangle \big|\leq 2C_x t^2
\qquad\mbox{for all }\; |y|, |\tilde y|=1\mbox{ and } |t|<\eta_x.
\end{equation}
which yields (through an application of Taylor's expansion):
\begin{lemma}\label{lem1}
Let $\phi\in C^{1,1}(x)$ be a bounded Borel function. Then $\{\Lse (x)\}_{\epsilon>0}$
are bounded independently of $\epsilon$ and $\Ls(x)$ is well
defined. More precisely, for all $\epsilon<\eta_x$ there holds:
\begin{equation}\label{uno}
|\Lse (x)|, \; |\Ls(x)| \leq 2c_s (1-s)\|\phi\|_{L^\infty}\eta_x^{-2s} + c_s s \cdot C_x\eta_x^{2-2s}.
\end{equation}
\end{lemma}
\begin{proof}
We use (\ref{C2}) to obtain, for any $|y|=|\tilde y|=1$:
\begin{equation*}
\begin{split}
\Big|\int_\epsilon^\infty \Lp(x, ty, t\tilde y)&\dm -
\int_\epsilon^{\eta_x} t\langle p_x, y-\tilde y\rangle \dm\Big| \leq \int_{\eta_x}^\infty 4\|\phi\|_{L^\infty}\dm
+ \int_\epsilon^{\eta_x}2 C_xt^2\dm \\ & = 2\frac{\alpha_s}{s\eta_x^{2s}}\|\phi\|_{L^\infty} +
\alpha_sC_x\frac{\eta_x^{2-2s}-\epsilon^{2-2s}}{1-s}.
\end{split}
\end{equation*}
On the other hand:
$$\sup_{|y|=1}\inf_{|\tilde y|\normalcolor =1} \int_\epsilon^{\eta_x} t\langle p_x, y-\tilde y\rangle \dm
= \int_\epsilon^{\eta_x} t\dm\cdot\sup_{|y|=1}\inf_{|\tilde y|\normalcolor=1} \langle p_x, y-\tilde y\rangle =0.$$
This results in:
\begin{equation*}
\begin{split}
|\Lse (x) | & = \Big|\Lse (x) - \sup_{|y|=1}\inf_{|\tilde y|\normalcolor =1}
\int_\epsilon^{\eta_x} t\langle p_x, y-\tilde y\rangle \dm\Big | \\ & \leq
\sup_{|y|=|\tilde y|=1} \Big|\int_\epsilon^\infty \Lp(x, ty, t\tilde y)\dm -
\int_\epsilon^{\eta_x} t\langle p_x, y-\tilde y\rangle \dm\Big|,
\end{split}
\end{equation*}
which proves the bound for $|\Lse (x) |$. The bound for $|\Ls (x)|$ follows similarly.
\end{proof}

Our proofs throughout the paper largely depend on analyzing the
behaviour of approximate extremizers $y, \tilde y$ in the definition (\ref{Ldefi})
We now observe that for the operator $\Ls$ these extremizers are
explicit, for a generic function $\phi$.

\begin{proposition}\label{prop2}
Let $\phi\in C^{1,1}(x)$ be a bounded Borel function such that $p_x\neq 0$. Then:
\begin{equation}\label{eq:equivdefifl}
\Ls(x) = \int_0^\infty \Lp\big(x, t\frac{p_x}{|p_x|}, t\frac{p_x}{|p_x|}\big)\dm
\end{equation}
\end{proposition}
\begin{proof}
For $\delta>0$, let $y_\delta$ be such that $\Ls(x)\leq \inf_{|\tilde y|=1}
\int_0^\infty \Lp(x, ty_\delta, t\tilde y)\dm +\delta$. Splitting the
integral and applying (\ref{C2}) on the interval $(0,\eta_x)$ yields:
\begin{equation*}
\begin{split}
\Ls (x)&\leq \int_0^\infty \Lp\big(x, ty_\delta,
t\frac{p_x}{|p_x|}\big)\dm +\delta \leq \int_{0}^{\eta_x} t\big\langle
p_x, y_\delta - \frac{p_x}{|p_x|}\big\rangle \dm + C+\delta,
\end{split}
\end{equation*}
where the constant $C$ depends on $s$ and $\phi$. Using the lower
bound in (\ref{uno}), we conclude that:
$$\big\langle p_x,\frac{p_x}{|p_x|} - y_\delta \big\rangle \cdot \int_{0}^{\eta_x}
t\dm < \infty.$$
Since the product above is nonnegative while the integral
diverges to $\infty$, we get $y_\delta=\frac{p_x}{|p_x|}$ and:
\begin{equation}\label{due}
\Ls(x) = \inf_{|\tilde y|=1} \int_0^\infty \Lp\big(x, t\frac{p_x}{|p_x|}, t\tilde y)\dm.
\end{equation}

Let now $\tilde y_\delta$ be such that $\Ls(x) \geq \int_0^\infty
\Lp\big(x, t\frac{p_x}{|p_x|}, t\tilde y_\delta)\dm -\delta$. As
before, in virtue of (\ref{C2}) we get:
$\Ls (x)\geq \int_{0}^{\eta_x} t\big\langle
p_x, \frac{p_x}{|p_x|}-\tilde y_\delta\big\rangle \dm - C -\delta$, so
the upper bound in (\ref{uno}) gives:
$$\big\langle p_x,\frac{p_x}{|p_x|} -\tilde y_\delta \big\rangle \cdot \int_{0}^{\eta_x}
t\dm < \infty.$$
Consequently $\tilde y_\delta=\frac{p_x}{|p_x|}$, so that: $\Ls(x) \geq \int_0^\infty
\Lp\big(x, t\frac{p_x}{|p_x|},  t\frac{p_x}{|p_x|})\dm -\delta$ for
all $\delta>0$. The proof is done, in view of (\ref{due}).
\end{proof}

Note that formulation (\ref{eq:equivdefifl}) corresponds, up to a
constant, to the one given in \cite[Definition 1.1]{BCF} for
$p_x\not=0$.  The case $p_x=0$ in that definition is equivalent to (\ref{inno}).  

\section{A proof of Theorem \ref{jaszczu1}}\label{sec3}

We first observe that taking $p_x=0$ in (\ref{C2}) implies the
following bound, for all $\epsilon<\eta_x$:
\begin{equation}\label{due.5}
\begin{split}
|\Lse (x)-\Ls(x)|&\leq \sup_{|y|=|\tilde y|=1} \big|\int_\epsilon^\infty \Lp(x, ty,
t\tilde y)\dm - \int_0^\infty \Lp(x, ty, t\tilde y)\dm \big|  \\ &\leq
\int_0^\epsilon 2C_x t^2\dm = c_s s\cdot C_x\epsilon^{2-2s}.
\end{split}
\end{equation}
In order to estimate the same difference when $p_x\neq 0$, we will
quantify estimates in the proof of Proposition \ref{prop2} for higher regular
functions, as specified below.

\begin{proposition}\label{kokon}
Assume (\ref{H1}) with $p_x\neq 0$. Then, for every $\epsilon< \eta_x$ there holds:
\begin{equation}\label{tre}
\begin{split}
\Big|\Lse (x)- & \int_\epsilon^\infty\Lp\big(x, t\frac{p_x}{|p_x|},
t\frac{p_x}{|p_x|}\big)\dm\Big|\\ & \leq 4 c_s s\cdot C_x \big(\eta_x^{2-2s} -
  \epsilon^{2-2s}\big)\cdot A_\epsilon + c_s (1-s) \cdot \Big( \eta_x^{-2s} +
\frac{2s}{2s-1}\eta_x^{1-2s}\Big)\cdot \omega_\phi(A_\epsilon),
\end{split}
\end{equation}
\end{proposition}
\begin{proof}
{\bf 1.} For every $\epsilon<\eta_x$ and every small $\delta>0$ 
let $|y_\delta^\epsilon|=1$ satisfy:
$\sup_{|y|=1}\int_\epsilon^\infty\phi(x+ty)\dm \leq \int_\epsilon^\infty\phi(x+ty_\delta^\epsilon)\dm +\delta$.
In particular, this implies:
$$\int_\epsilon^\infty\phi\big(x+t\frac{p_x}{|p_x|}\big)\dm \leq
\int_\epsilon^\infty\phi(x+ty_\delta^\epsilon)\dm +\delta,$$ 
Denote $A= \Big|\frac{p_x}{|p_x|} - y_\delta^\epsilon\Big|$.  Together
with (\ref{C11}), the above bound results in:
\begin{equation*}
\begin{split}
\delta &\geq \int_\epsilon^{\eta_x} \big(\phi\big(x+t\frac{p_x}{|p_x|}\big)
- \phi(x+t y_\delta^\epsilon) \big)\dm - \int_{\eta_x}^\infty \big|\phi\big(x+t\frac{p_x}{|p_x|}\big)
- \phi(x+t y_\delta^\epsilon)\big|\dm \\ & \geq 
\int_\epsilon^{\eta_x} t\Big\langle\nabla\phi\big(x+t\frac{p_x}{|p_x|}\big),
\frac{p_x}{|p_x|}-y_\delta^\epsilon\Big\rangle\dm -
\int_\epsilon^{\eta_x} C_xt^2A^2\dm 
- \int_{\eta_x}^\infty (1+t)\cdot\omega_\phi (A)\dm
\\ & \geq \Big\langle p_x,\frac{p_x}{|p_x|} - y_\delta^\epsilon\Big\rangle
\int_\epsilon^{\eta_x} t\dm - 4C_xA \int_\epsilon^{\eta_x} t^2\dm  -\omega_\phi(A) \int_{\eta_x}^\infty (1+t)\dm.
\end{split}
\end{equation*}
The last bound above follows by observing that $\big|\nabla\phi\big(x+t\frac{p_x}{|p_x|})- \nabla
\phi(x)\big|\leq 2C_xt$ for all $|t|\leq \eta_x$ and that $A\leq 2$. Consequently, we get:
\begin{equation*}
\begin{split}
\Big\langle &p_x,\frac{p_x}{|p_x|} - y_\delta^\epsilon\Big\rangle\leq  \frac{1}{
\int_\epsilon^{\eta_x} t\dm}\Big( \delta+ 4C_xA
\int_\epsilon^{\eta_x} t^2\dm +\omega_\phi(A)\int_{\eta_x}^\infty (1+t)\dm\Big).
\end{split}
\end{equation*}

On the other hand, by a straightforward calculation:
\begin{equation*}
\begin{split}
A^2=\Big|\frac{p_x}{|p_x|} - y_\delta^\epsilon\Big|^2 = 2 - 2 \Big\langle
\frac{p_x}{|p_x|},  y_\delta^\epsilon\Big\rangle = \frac{2}{|p_x|}
\Big\langle p_x,\frac{p_x}{|p_x|} - y_\delta^\epsilon\Big\rangle,
\end{split}
\end{equation*}
the last two displayed formulas yields that:
\begin{equation}\label{quattro}
\begin{split}
A^2 \leq  \frac{2}{|p_x| \int_\epsilon^{\eta_x} t\dm}\Big( \delta+ 4C_xA
\int_\epsilon^{\eta_x} t^2\dm +\omega_\phi(A) \int_{\eta_x}^\infty (1+t)\dm\Big).
\end{split}
\end{equation}

We now simplify (\ref{quattro}) as follows. Without loss of
generality, we may assume that $\delta>0$ satisfies: 
$ \delta \cdot |p_x|\int_\epsilon^{\eta_x}t\dm\leq \big(16 C_x \int_\epsilon^{\eta_x}t^2\dm\big)^2,$
In case when $\delta$ is larger than the two other terms in the right
hand side of (\ref{quattro}), we get:
\begin{equation}\label{cinque}
A^2\leq \frac{4\delta}{|p_x|\int_\epsilon^{\eta_x} t\dm}\leq
\Big(\frac{32 C_x \int_\epsilon^{\eta_x}t^2\dm}{|p_x|\int_\epsilon^{\eta_x}t\dm}\Big)^2,
\end{equation}
In the opposite case, there holds:
\begin{equation*}
A^2 \leq  \frac{4}{|p_x| \int_\epsilon^{\eta_x} t\dm}\Big(4C_xA
\int_\epsilon^{\eta_x} t^2\dm +\omega_\phi(A) \int_{\eta_x}^\infty
(1+t)\dm\Big)\doteq I_1 + I_2.
\end{equation*}
Further, when $I_2\leq I_1$, then we obtain the same bound as in (\ref{cinque}), namely:
\begin{equation*}
A\leq \frac{32 C_x  \int_\epsilon^{\eta_x}t^2\dm}{|p_x|\int_\epsilon^{\eta_x}t\dm} = 
\frac{16 C_x}{|p_x|}\cdot\frac{2s-1}{1-s}\cdot 
\frac{\eta_x^{2-2s}-\epsilon^{2-2s}}{\epsilon^{1-2s} -{\eta_x}^{1-2s}}. 
\end{equation*}
On the other hand, $I_1<I_2$ implies:
\begin{equation*}
A^2 \leq  \frac{8 \omega_\phi(A) \int_{\eta_x}^\infty (1+t)\dm}{|p_x|  \int_\epsilon^{\eta_x} t\dm}
=\frac{8\omega_\phi(A)}{|p_x|}\cdot \frac{\frac{2s-1}{2s}
  \eta_x^{-2s} +\eta_x^{1-2s}}{\epsilon^{1-2s} - \eta_x^{1-2s}}.
\end{equation*}
We hence conclude that $A\leq A_\epsilon$ in either of the above cases.

\smallskip

{\bf 2.} Similarly as in step 1, we see that the unit vector  $\tilde
y_\delta^\epsilon$ with the property:
$\inf_{|y|=1}\int_\epsilon^\infty\phi(x-ty)\dm \geq
\int_\epsilon^\infty\phi(x-t\tilde y_\delta^\epsilon)\dm -\delta$,
satisfies: $\big|\frac{p_x}{|p_x|}-\tilde y_\delta^\epsilon\big|\leq
A_\epsilon$. We now write:
\begin{equation*}
\int_\epsilon^\infty \Lp(x, t\tilde y_\delta^\epsilon, t\tilde  y_\delta^\epsilon)\dm
-\delta \leq \Lse (x)\leq \int_\epsilon^\infty \Lp(x,
ty_\delta^\epsilon, ty_\delta^\epsilon)\dm +\delta,
\end{equation*}
which implies:
\begin{equation}\label{sette}
\begin{split}
& \Big|\Lse (x) - \int_\epsilon^\infty \Lp\big(x, t\frac{p_x}{|p_x|}, t\frac{p_x}{|p_x|}\big)\dm
\Big|\leq \delta + \max \big\{ |I(y_\delta^\epsilon)|, |I(\tilde y_\delta^\epsilon)|\big\},\\ &
\mbox{where: } I(y)\doteq \int_\epsilon^\infty \phi(x+ty) - \phi\big(x+t \frac{p_x}{|p_x|}\big) + 
\phi(x-ty) - \phi\big(x-t \frac{p_x}{|p_x|}\big) \dm.
\end{split}
\end{equation}
Observe that:
\begin{equation*}
\begin{split}
|I(y_\delta^\epsilon)|\leq & \; \Big|\int_\epsilon^{\eta_x} \phi(x+t
y_\delta^\epsilon) - \phi\big(x+t \frac{p_x}{|p_x|}\big)  + 
\phi(x-t y_\delta^\epsilon) - \phi\big(x-t \frac{p_x}{|p_x|}\big)\dm\Big| \\ & + 
\int_{\eta_x}^\infty \big|\phi(x+t y_\delta^\epsilon) - \phi\big(x+t \frac{p_x}{|p_x|}\big)\big| + 
\big|\phi(x-t y_\delta^\epsilon) - \phi\big(x-t
\frac{p_x}{|p_x|}\big)\big|\dm \doteq \bar I_1 + \bar I_2.
\end{split}
\end{equation*}
In order to deal with $\bar I_1$, we use the Taylor expansion:
\begin{equation*}
\Big|\phi(x\pm t y_\delta^\epsilon) - \phi\big(x\pm t \frac{p_x}{|p_x|}\big)
-\Big\langle \nabla\phi\big(x\pm t \frac{p_x}{|p_x|}\big), \pm t
\big(y_\delta^\epsilon - \frac{p_x}{|p_x|}\big)\Big\rangle \Big|\leq
C_x t^2  \big| y_\delta^\epsilon - \frac{p_x}{|p_x|}\big|^2,
\end{equation*}
which upon integration implies:
\begin{equation*}
\begin{split}
\bar I_1 & \leq \int_\epsilon^{\eta_x} C_x t^2 \big(4 A+ 2A^2\big)\dm\leq
8C_x A\int_\epsilon^{\eta_x} t^2\dm = 4 C_x \alpha_s \frac{\eta_x^{2-2s} - \epsilon^{2-2s}}{1-s}\cdot A.
\end{split}
\end{equation*}
For the term $\bar I_2$, we get:
\begin{equation*}
\begin{split}
\bar I_2 & \leq 2\int_{\eta_x}^\infty (1+t)\cdot\omega_\phi(A)\dm
= 2 \alpha_s \Big( \frac{\eta_x^{-2s}}{2s} + \frac{\eta_x^{1-2s}}{2s-1}\Big)\cdot \omega_\phi(A).
\end{split}
\end{equation*}

In conclusion, we obtain the following bounds:
\begin{equation*}
|I(y_\delta^\epsilon)|, \quad |I(\tilde y_\delta^\epsilon)|\leq 4 C_x \alpha_s \frac{\eta_x^{2-2s} -
  \epsilon^{2-2s}}{1-s} A_\epsilon + 2 \alpha_s \Big( \frac{\eta_x^{-2s}}{2s} +
\frac{\eta_x^{1-2s}}{2s-1}\Big)\cdot \omega_\phi(A_\epsilon).
\end{equation*}
This ends the proof in virtue of (\ref{sette}).
\end{proof}

\medskip

\begin{corollary}\label{cor4.5}
Under the same assumptions and notation as in Proposition \ref{kokon}, we have:
\begin{equation*}
\begin{split}
\Big|&\Lse (x)- \Ls (x)\Big|\\ &\quad \leq 4 c_s s \cdot C_x \big(\eta_x^{2-2s} -
  \epsilon^{2-2s}\big) \cdot A_\epsilon + c_s (1-s)\cdot \Big(\eta_x^{-2s} +
\frac{2s}{2s-1}\eta_x^{1-2s}\Big)\cdot \omega_\phi(A_\epsilon) + c_s s\cdot C_x\epsilon^{2-2s}.
\end{split}
\end{equation*}
\end{corollary}
\begin{proof}
Observe that for all $t<\eta_x$ there holds:
$$\big|\Lp\big(x, t\frac{p_x}{|p_x|}, t\frac{p_x}{|p_x|}\big) \big|
\leq t^2 \|\nabla^2\phi\|_{L^\infty(B_t)}.$$
Consequently and in view of Proposition \ref{prop2} we get:
$$\Big|\Ls (x) - \int_\epsilon^\infty \Lp\big(x, t\frac{p_x}{|p_x|}, t\frac{p_x}{|p_x|}\big)\dm\Big|
\leq \int_0^\epsilon \Big|\Lp\big(x, t\frac{p_x}{|p_x|}, t\frac{p_x}{|p_x|}\big)\Big|\dm
\leq \frac{\alpha_s}{1-s}\cdot C_x \epsilon^{2-2s}.$$
This achieves the proof by Proposition \ref{kokon}.
\end{proof}

\medskip

Note that the bound in Corollary \ref{cor4.5} is essentially valid in both cases $p_x\neq 0$ and
$p_x=0$, because of (\ref{due.5}). Scaling the said bound by the
factor $\frac{s\epsilon^{2s}}{\alpha_s}$, we directly deduce Theorem \ref{jaszczu1}.

\medskip

\begin{remark}\label{kura}
\begin{enumerate}[leftmargin=7mm]
\item [(i)] Observing that: $\omega_\phi(a)\leq 2\|\phi\|_{L^\infty}$,
  we get for all $\epsilon<\frac{\eta_x}{2}$:
$$\kappa_\epsilon\leq \Big(\frac{16\|\phi\|_{L^\infty}}{|p_x|}\cdot 
\frac{\frac{2s-1}{2s} \eta_x^{-2s} +\eta_x^{1-2s}}{\epsilon^{1-2s} - \eta_x^{1-2s}}\Big)^{1/2}
\leq 8 \Big(\frac{\|\phi\|_{L^\infty}}{|p_x|}\cdot 
\frac{\eta_x^{-2s} +\eta_x^{1-2s}}{2s-1}\Big)^{1/2}\epsilon^{s-1/2}.$$
In the second inequality we used that for all $\epsilon<\frac{\eta_x}{2}$
and all $s\in (\frac{1}{2},1)$ there holds:
$$\epsilon^{1-2s} -
\eta_x^{1-2s}> \epsilon^{1-2s} (1-2^{1-2s}), \qquad 1-2^{1-2s}\geq (2s-1)\ln\sqrt{2} >\frac{2s-1}{4}.$$
The first quantity in $A_\epsilon$ is of order
$\epsilon^{2s-1}$, so the right hand side in (\ref{tre}) is:
$$ C(s)\cdot C\big(C_x, \frac{1}{|p_x|}, \|\phi\|_{L^\infty}\big)\cdot
C(\eta_x)\epsilon^{s-1/2} + C(s)\cdot C\big(\frac{1}{|p_x|}, \|\phi\|_{L^\infty}\big)\cdot
C(\eta_x)\omega_\phi(\epsilon^{s-1/2}), $$
where $C(s)$ depends only on $s$ and $C(\eta_x)$ only on $\eta_x$ and
the remaining constants depend on the other displayed terms, in a
nondecreasing manner.
\item[(ii)] When $\phi\in C^{0,\alpha}(\R^N\setminus \bar B_{\eta_x})$ with
$\alpha\in (0,1)$, then $\omega_\phi(a)=[\phi]_\alpha a^{\alpha}$. Therefore:
$$\kappa_\epsilon \leq \Big(\frac{32\;[\phi]_{\alpha}}{|p_x|}\cdot 
\frac{\eta_x^{-2s} +\eta_x^{1-2s}}{2s-1}\Big)^{\frac{1}{2-\alpha}}\epsilon^{\frac{2s-1}{2-\alpha}},$$
whereas (\ref{tre}) can be replaced with:
$ C(s)\cdot C\big(\frac{1}{|p_x|}, [\phi]_{\alpha}\big)\cdot
C(\eta_x) \epsilon^{\alpha\cdot\frac{2s-1}{2-\alpha}}.$
\item[(iii)] Finally, for $\phi$  Lipschitz on $\R^N\setminus \bar
  B_{\eta_x}$ with the Lipschitz constant $\mbox{Lip}_\phi$, we get: 
$$\kappa_\epsilon \leq \frac{32 \; \mbox{Lip}_\phi}{|p_x|}\cdot 
\frac{\eta_x^{-2s} +\eta_x^{1-2s}}{2s-1}\epsilon^{2s-1},
\quad A_\epsilon\leq \frac{32}{|p_x|}\cdot \max\Big\{\frac{2C_x\eta_x^{2-2s}}{1-s},
\frac{\mbox{Lip}_\phi (\eta_x^{-2s}+\eta_x^{1-2s})}{2s-1}\Big\}\epsilon^{2s-1}.$$
Indeed, both quantities in $A_\epsilon$ have $\epsilon^{2s-1}$-order.
The expression in (\ref{tre}) is then:
$ C(s)\cdot C\big(C_x,\frac{1}{|p_x|}, \mbox{Lip}_\phi\big)\cdot
C(\eta_x) \epsilon^{2s-1}$, whereas the order of the error
bounding quantity in Theorem \ref{jaszczu1} is
$C(s)\cdot(\epsilon^{4s-1}+\epsilon^2)$ as $\epsilon\to 0+$, and
$C(s)\to\infty$ as $s\to 1-$.
\end{enumerate}
\end{remark}

\section{A proof of Theorem \ref{jaszczu2}}\label{sec4}

We note the following refinement of the argument in the proof of Corollary \ref{cor4.5}:

\begin{proposition}\label{grazyna}
Let $\phi\in C^2(\bar B_{\eta_x})$ satisfy: $p_x\doteq \nabla
\phi(x)\neq 0$. Then, for every $\epsilon< \eta_x$ such that $\epsilon
|\nabla^2\phi(x)|\leq |p_x|$, there holds:
\begin{equation}\label{sette.5}
\begin{split}
\Big| c_s s\cdot \epsilon^{-2s} \cdot \frac{1}{2}\Big(\sup_{B_\epsilon(x)}&\;\phi +
 \inf_{B_\epsilon(x)}\phi -2\phi(x)\Big) - \int_0^\epsilon\Lp\big(x,
t\frac{p_x}{|p_x|}, t\frac{p_x}{|p_x|}\big)\dm \Big| \\ & \leq 
c_s s\Big(2\epsilon^{3-2s}\frac{|\nabla^2\phi(x)|^2}{|p_x|}
+ \epsilon^{2-2s}\sup_{y\in B_\epsilon(x)}|\nabla^2\phi(y) - \nabla^2\phi(x)|\Big).
\end{split}
\end{equation}
\end{proposition}
\begin{proof}
A simple application of Taylor's expansion yields:
$$\big|\Lp\big(x, t\frac{p_x}{|p_x|}, t\frac{p_x}{|p_x|}\big) - t^2\Delta_\infty\phi(x)\big|
\leq t^2 \sup_{y\in B_t}|\nabla^2\phi(y) - \nabla^2\phi(x)|,$$
where we recall that $\Delta_\infty\phi(x)=\big\langle
\nabla^2\phi(x): \frac{p_x}{|p_x|}\otimes \frac{p_x}{|p_x|}\big\rangle$. 
Integrating the above $\int_0^\epsilon\dm$, we get:
$$ \Big|\int_0^\epsilon\Lp\big(x, t\frac{p_x}{|p_x|}, t\frac{p_x}{|p_x|}\big) \dm - 
\frac{\alpha_s}{2(1-s)}\epsilon^{2-2s}\Delta_\infty\phi(x)\Big|\leq \frac{\alpha_s}{2(1-s)}\epsilon^{2-2s}
\sup_{y\in B_\epsilon}|\nabla^2\phi(y) - \nabla^2\phi(x)|.$$
Recalling that (see for example \cite[section 3.2]{LewB}):
\begin{equation}\label{otto}
\Big|\big( \sup_{B_\epsilon} \phi + \inf_{B_\epsilon} \phi
-2\phi(x)\big) - \epsilon^2\Delta_\infty\phi(x)\Big|\leq 4\epsilon^3\frac{|\nabla^2\phi(x)|^2}{|p_x|}
+ \epsilon^2 \sup_{y\in B_\epsilon}|\nabla^2\phi(y) - \nabla^2\phi(x)|,
\end{equation}
and taking the linear combination of the two above formulas, the proof is done.
\end{proof}

\medskip

The proof of Theorem \ref{jaszczu2} follows directly by summing up
formulas (\ref{tre}), (\ref{sette.5}), and multiplying the result by the factor $\frac{(1-s)s}{\alpha_s}\epsilon^{2s}$. Since:
$$\Big(\frac{(1-s)s}{\alpha_s}\epsilon^{2s}\Big) \Lse (x) +\Big(\frac{(1-s)s}{\alpha_s}\epsilon^{2s}\Big)\cdot
\frac{\alpha_s}{2(1-s)}\epsilon^{-2s}\big(\sup_{B_\epsilon}\phi
+\inf_{B_\epsilon}\phi - 2\phi(x)\big) = \A_\epsilon\phi(x) - \phi(x),$$
the error in the claimed expansion is the sum of errors
in (\ref{tre}) and (\ref{sette.5}), multiplied by $\frac{(1-s)s}{\alpha_s}\epsilon^{2s}$.

\begin{remark}\label{piesio}
\begin{enumerate}[leftmargin=7mm]
\item[(i)] Analysis similar to Remark \ref{kura} allows for computing
  the order of the error term in Theorem \ref{jaszczu2} when $\phi$ is Lipschitz:
\begin{equation*}
\begin{split}
&C(s) \cdot C\big(C_x,\frac{1}{|p_x|}, \mbox{Lip}_\phi\big)\cdot C(\eta_x)
\epsilon^{4s-1} + C(s)\cdot C\big(|\nabla^2\phi(x)|,
\frac{1}{|p_x|}\big)\epsilon^3 +C(s)\cdot o(\epsilon^2).
\end{split}
\end{equation*}
As before, $C(s)$ depends only on $s$, and $C(\eta_x)$ only on
$\eta_x$, while the remaining constants depend on the displayed terms in a
nondecreasing manner. For $\phi\in C^{2,1}(B_{\eta_x})$,
the above quantity has order $\epsilon^{4s-1}+\epsilon^3$, which
equals $\epsilon^3$ at $s=1$.

\item[(ii)] For a more precise analysis of the asymptotic expansion when $s\to 1-$, note that:
\begin{equation*}
\begin{split}
& \kappa_\epsilon\leq \sup\Big\{ a;~ a\in [0,2] \; \mbox{ and } \; a^2\leq
 \frac{32\;\omega_\phi(a)}{|p_x|}\cdot \frac{
  \eta_x^{-2s} +\eta_x^{1-2s}}{2s-1}\epsilon^{2s-1}\Big\}, \\ &
\frac{16 C_x}{|p_x|}\cdot\frac{2s-1}{1-s}\cdot 
\frac{\eta_x^{2-2s}-\epsilon^{2-2s}}{\epsilon^{1-2s} -{\eta_x}^{1-2s}}
\leq \frac{16 C_x}{|p_x|}\cdot 16\eta_x^{2-2s}|\ln\epsilon|\epsilon^{2s-1}.
\end{split}
\end{equation*}
The first bound above is valid when $\epsilon<\frac{\eta_x}{2}$, while
for the second bound we used:
$\eta_x^{2-2s}-\epsilon^{2-2s}\leq (2-2s) (\ln\eta_x - \ln\epsilon)
\eta_x^{2-2s}\leq 4(1-s)|\ln\epsilon|\eta_x^{2-2s}$, when $\epsilon<e^{-|\ln \eta_x|}$.
Consequently, $A_\epsilon\leq o(1)$ as $\epsilon\to 0+$, uniformly in  
$s\in (\frac{1}{2}+\delta,1)$. For each fixed $\epsilon$, the bound in
Theorem \ref{jaszczu2} converges to (consistently with (\ref{otto}) as $s\to 1-$): 
$$2\epsilon^3\frac{|\nabla^2\phi(x)|^2}{|p_x|}
+ \epsilon^2\sup_{y\in  B_\epsilon}\big|\nabla^2\phi(y)-\nabla^2\phi(x)\big|.$$
We also observe that when $\phi$ is Lipschitz on $\R^N\setminus \bar B_{\eta_x}$, the
said bound becomes:
\begin{equation*}
\begin{split} 
& (1-s) \epsilon^{4s-1}|\ln\epsilon|^2\cdot
\frac{2^9 C_x}{|p_x|} \;  \cdot\Big(16s C_x\eta_x^{4-4s} +
\mbox{Lip}_\phi\cdot \big(\frac{\eta_x^{2-4s}}{2}+\frac{s\eta_x^{3-4s}}{2s-1}\big)|\ln\epsilon|^{-1}\Big)
\\ & + s\Big(2\epsilon^3\frac{|\nabla^2\phi(x)|^2}{|p_x|}
+ \epsilon^2\sup_{y\in B_\epsilon}\big|\nabla^2\phi(y)-\nabla^2\phi(x)\big|\Big).
\end{split}
\end{equation*}
\end{enumerate}
\end{remark}

\section{Further remarks}\label{sec5}

\subsection{Spherical prisms as integration domains}\label{sec52}
With an eye towards future applications, we now consider another averaging operator:
\begin{equation}\label{MVP3}
\bar\A_\epsilon^o \phi(x) = \frac{1}{2}\Big(\sup_{|y|=1}\fint_{T^{\epsilon, R, \alpha}(y)}
\phi(x+z)\dmN + \inf_{|y|=1}\fint_{T^{\epsilon, R, \alpha}(y)}
\phi(x+z)\dmN \Big).
\end{equation}
Above, the integration is taken with respect to the measure $\mu_s^N$ on the Borel subsets of $\R^N$:
$$\mathrm{d}\mu_s^N(z)  \doteq \frac{C(N,s)}{|z|^{N+2s}}\;\mathrm{d}z \quad \mbox{ where }\;
C(N,s)=\frac{4^s s\Gamma\big(\frac{N}{2}+s\big)}{\pi^{N/2}\Gamma\big(1-s\big)} =
\Big(\int_{\R^N}\frac{1-\cos \langle z, e_1\rangle}{|z|^{N+2s}}\;\mathrm{d}z\Big)^{-1} .$$
Clearly, $C(1,s)=\alpha_s$ and $\mu_s^1=\mu_s$. The integration domain 
$T^{\epsilon, R, \alpha}(y)$ is the regular spherical prism in $\R^N$,
oriented in the direction $y\in\R^N\setminus\{0\}$, truncated at the
heights $0<\epsilon<R$, and with the
aperture angle $\angle$ determined by $\alpha>0$ as described in:
$$T^{\epsilon, R, \alpha}(y) = \Big\{ z\in\R^N; ~ \sin \frac{\angle
  (y,z)}{2}<\alpha, ~~ \langle y,z\rangle>0 \;\mbox{ and } \; \epsilon < |z|<R\Big\}.$$
With the above notation, $T^{0, \infty, \alpha}(y)$ is an infinite
cone, and we observe that such cones were used in the definition of
the fractional $p$-Laplacian $\Delta_p^s$ in \cite{BCF2}, with
$p=p(\alpha, N, s)$. We have:

\begin{lemma}\label{lem8}
Assume (\ref{H1}). Then, for every $\epsilon<\eta_x$, $R>\max\{\eta_x, 1\}$
and $\alpha<\frac{1}{2}$, there holds:
\begin{equation*}
\begin{split}
&\sup_{|y|=1} \Big|\fint_{T^{\epsilon, R, \alpha}(y)} \phi(x+z)\dmN - \fint_\epsilon^\infty
\phi(x+ty)\dm \Big| \\ & \qquad \qquad \leq 2 \big|\frac{\epsilon}{R}\big|^{2s} \cdot \|\phi\|_{L^\infty}
+ \max\Big\{2\big(|p_x| + 2C_x\eta_x \big)\eta_x\cdot \alpha,\;
3R\cdot \omega_\phi(\alpha)\Big\}. 
\end{split}
\end{equation*}
\end{lemma}

\begin{proof}
We first estimate the difference:
\begin{equation*}
\begin{split}
\Big|\fint_\epsilon^\infty& \phi(x+ty)\dm - \fint_\epsilon^R
\phi(x+ty)\dm \Big| \\ & \leq \frac{1}{\mu_s(\epsilon, \infty)} \int^\infty_{R} |\phi(x+ty)|\dm
+ \Big|\frac{1}{\mu_s(\epsilon, \infty)} - \frac{1}{\mu_s(\epsilon,
  R)}\Big|  \int_\epsilon^{R} |\phi(x+ty)|\dm \\ &
\leq 2 \big|\frac{\epsilon}{R}\big|^{2s} \cdot \|\phi\|_{L^\infty}.
\end{split}
\end{equation*}
Next, observe that:
\begin{equation*}
\begin{split}
\int_{T^{\epsilon, R, \alpha}(y)} \phi(x+|z|y)\dmN & =
\int_\epsilon^R\phi(x+ty) t^{N-1} \cdot\mbox{area}\big(\big\{|z|=1,~
z\in T^{\epsilon, R, \alpha}\big\}\big) \frac{\mbox{d}t}{t^{N+2s}}
\\ & = \mbox{area}\big(\big\{|z|=1,~
z\in T^{\epsilon, R, \alpha}\big\}\big) \cdot \int_{\epsilon}^R\phi(x+ty)\dm
\end{split}
\end{equation*}
which implies:
$\displaystyle \fint_{T^{\epsilon, R, \alpha}(y)} \phi(x+|z|y)\dmN = \fint_{\epsilon}^R\phi(x+ty)\dm.$
It remains to bound:
\begin{equation*}
\begin{split}
\fint_{T^{\epsilon, R, \alpha}(y)} & |\phi(x+z) -\phi(x+|z|y)|\dmN \leq 
\sup_{z\in T^{\epsilon, R, \alpha}(y)}  |\phi(x+z) -\phi(x+|z|y)| \\ & \leq 
\max\Big\{\|\nabla \phi\|_{L^\infty(B_{\eta_x})} \cdot
2\eta_x\cdot\alpha,\; \omega_\phi(2 R\alpha)\Big\} \\ &
\leq \max\Big\{2\big(|p_x| + 2C_x\eta_x \big)\eta_x\cdot \alpha,\;
(1+2R)\cdot \omega_\phi(\alpha)\Big\}. 
\end{split}
\end{equation*}
This yields the desired estimate and ends the proof.
\end{proof}

From Lemma \ref{lem8}, Remark \ref{kura} and Theorem \ref{jaszczu1}, we directly deduce:

\begin{corollary}\label{cor9}
Assume (\ref{H1}) with $\eta_x\leq 1$, and that $\phi$ is
Lipschitz on $\R^N\setminus \bar B_{\eta_x}$ with Lipschitz constant
$\mbox{Lip}_\phi$.  For every $\epsilon\ll \eta_x$,  we set
$R=\epsilon^{\frac{1}{2s}-1}$ and $\alpha=\epsilon^{4s-\frac{1}{2s}}$. 
Then there holds:
\begin{equation*}
\begin{split}
\Big|\bar\A_\epsilon^o\phi(x) & - \phi(x) - \frac{1}{c_s(1-s)}\epsilon^{2s}\Ls (x)\Big|
 \leq  \epsilon^{4s-1}\big(2\|\phi\|_{L^\infty} +
3\mbox{Lip}_\phi\big) + \frac{s}{1-s}\cdot 2C_x\epsilon^2 
\\ & \qquad\qquad + \left\{\begin{array}{ll} {\displaystyle
 \frac{32}{|p_x|}\;\epsilon^{4s-1}\Big(\frac{8s}{1-s}
 + \big(\eta_x^{-2s}+\frac{2s}{2s-1}\eta_x^{1-2s}\big){\mbox{Lip}_\phi} 
\Big)\;\cdot } & \vspace{1mm}\\ \qquad\quad\ {\displaystyle \cdot
\max\Big\{\frac{2C_x}{1-s}, \frac{(\eta_x^{-2s}+\eta_x^{1-2s})}{2s-1}\mbox{Lip}_\phi\Big\}}
& \mbox{ when } p_x\neq 0 \vspace{1.5mm} \\ 0 & \mbox{ when } p_x= 0.
\end{array}\right.
\end{split}
\end{equation*}
\end{corollary}

\begin{remark}\label{discrete}
Towards the applications in the numerical approximating of solutions to the nonlocal
Dirichlet problem for the operator $\Delta_\infty^s$, one has to consider a
discrete version of the result in Theorem \ref{jaszczu1}. 
To this end, let $\{\theta_i\}_{i=1}^n$ be an equidistributed spherical grid on
$\{|z|=1\}\subset\R^N$; when $N=2$ then $\theta_i=e^{2\pi i/n}$.
Next, for all $x_k$ in the cubical grid $h \mathbb{Z}^N$ define:
$$\bar\A_\epsilon^d \phi(x_k) \doteq \frac{s h^N}{|S_\alpha|\big(\epsilon^{-2s}-R^{-2s}\big)}
\cdot \big(\max_{i=1\ldots n} + \min_{i=1\ldots n} \big)
\sum_{x_j\in  T^{\epsilon, R,\alpha}(\theta_i)\cap h\mathbb{Z}^N}\frac{\phi(x_{k}+x_j)}{|x_j|^{N+2s}},$$ 
where we used that $\mu_s^N(T^{\epsilon, R,\alpha}) = C(N,s) |S_\alpha|\cdot\frac{\epsilon^{-2s}-R^{-2s}}{2s}$,
with $S_\alpha \doteq T^{0,\infty,\alpha}\cap \{|z|=1\}.$

It is clear that for $h$ and $n$ scaling in $\epsilon$ with
sufficiently high positive and negative powers, respectively, the averaging operator $\bar\A_\epsilon^d$
is a discrete approximation of $\mathcal{A}_\epsilon^o$ at the same rate of the error proved in Corollary \ref{cor9}.
The details of this construction as well as its implementation for a
numerical scheme, are left for the future work.
\end{remark}

\subsection{The viscosity framework}\label{sec51}
We observe that our results may be reformulated in the viscosity
setting, which has been used in the results of\cite{MPR10} for the (local) $\infty$- Laplacian.   
The definition of viscosity solutions for the fractional $\infty$-Laplacian
as in \cite[Definition 2.3]{BCF} encodes the hypothesis
(\ref{H1}) which needs to be satisfied by the test functions $\phi$. Following
this lead, one can consider the asymptotic expansions in the viscosity
sense. From now on, the respective averages in
(\ref{MVP1}), (\ref{MVP2}) and (\ref{MVP3}), are generically denoted
by $Average_\epsilon$, with corresponding constants $K>0$ such that:
$$ \frac{K}{\epsilon^{2s}}\big(\emph{Average}_\epsilon \phi-\phi\big)=
\Delta_\infty^s\phi + o(1) \qquad \textup{as $\;\epsilon\to 0+$.} $$

\begin{definition} \label{def:ViscMVP}
Let $\Omega\subset \R^N$ be open and let $f:\Omega\to\R$. A bounded
upper (resp. lower) semicontinuous function $u:\R^N\to\R$ is a viscosity sub-solution (resp. super-solution) of: 
\begin{equation}\label{vs}
\frac{K}{\epsilon^{2s}}\big(\emph{Average}_\epsilon u-u\big)=f +
o(1) \quad \text{in $\Omega\qquad $ as $\;\epsilon\to 0+$},
\end{equation}
provided that the following holds.  For every $x\in \Omega$, $r>0$, and $\psi\in C^2(\bar{B}_r(x))$ such that:
\[
\psi(x)=u(x) \quad \text{and} \quad \psi(y)>u(y) \quad
\text{(resp. $\psi(y)<u(y)$)} \quad \text{for all $y\in
  \bar{B}_{r}(x)\setminus\{x\}$,}   
\]
we have:
\[
\frac{K}{\epsilon^{2s}}\big(\emph{Average}_\epsilon \phi(x)-\phi(x)\big)\geq
f(x)+ o(1)\quad
\text{\Big(resp. $\frac{K}{\epsilon^{2s}}\big(\emph{Average}_\epsilon \phi(x)-\phi(x)\big)\leq f(x) + o(1)$\Big)}, 
\]
where $\phi \doteq \mathds{1}_{\bar B_r(x)}\psi + \mathds{1}_{\R^N\setminus \bar B_r(x)} u$.
When $u$ is both a viscosity sub- and super-solution, it is
a  viscosity solution of (\ref{vs}) (i.e. it satisfies the asymptotic expansion in the viscosity sense).
\end{definition}

The following follows from either of Theorems
\ref{jaszczu1}, \ref{jaszczu2} or Corollary \ref{cor9} in a standard fashion: 

\begin{theorem}\label{thm:equivalenceviscosity1}
Let $\Omega\subset \R^N$ be open, $u:\R^N\to\R$ be bounded and uniformly
continuous, and $f:\Omega\to\R$. Then the following are equivalent: 
\begin{enumerate}[leftmargin=7mm]
\item[(i)] $u$ is a viscosity solution of: $\Delta_\infty^s u=f$ in $\Omega$,\vspace{1mm}
\item[(ii)] $u$ satisfies: $\frac{K}{\epsilon^{2s}}\big(\text{Average}_\epsilon u-u\big)=f+
  o(1)$ in $\Omega$ as $\epsilon\to0+$, in the viscosity sense.
\end{enumerate}
\end{theorem}
We refer the reader to \cite{BS2} for similar statements in the
context of other averages for $\Delta_\infty^s$.

\section*{Acknowledgements}
This research was started when M. Lewicka visited Norwegian University of
Science and Technology. Part of this research was then carried out while
F. del Teso and J. Endal visited the University of Pittsburgh. We want
to thank both institutions for their hospitality.  

\end{document}